\documentclass[11pt,a4paper]{article}

\usepackage[T1]{fontenc}
\usepackage[utf8]{inputenc}
\usepackage{lmodern}
\usepackage[a4paper,margin=27mm]{geometry}
\usepackage{microtype}
\usepackage{amsmath,amssymb,amsthm,mathtools}
\usepackage{booktabs,longtable,tabularx,array}
\usepackage{enumitem}
\usepackage{graphicx}
\usepackage{tikz}
\usepackage{tikz-cd}
\usepackage{forest}
\usepackage{caption}
\usepackage{subcaption}
\usepackage{xcolor}
\usepackage{fancyhdr}
\usepackage{setspace}
\usepackage{csquotes}
\usepackage{hyperref}
\usepackage[nameinlink,noabbrev]{cleveref}

\hypersetup{
  colorlinks=true,
  linkcolor=black,
  citecolor=black,
  urlcolor=black,
  pdftitle={Repunit Coordinates for the Erdős Ternary-Digit Problem},
  pdfauthor={Michael A. Idowu}
}

\setlist[itemize]{leftmargin=1.6em,itemsep=0.2em,topsep=0.3em}
\setlist[enumerate]{leftmargin=1.8em,itemsep=0.2em,topsep=0.3em}
\newtheorem{theorem}{Theorem}[section]
\newtheorem{lemma}[theorem]{Lemma}
\newtheorem{proposition}[theorem]{Proposition}
\newtheorem{corollary}[theorem]{Corollary}
\theoremstyle{definition}
\newtheorem{definition}[theorem]{Definition}
\newtheorem{example}[theorem]{Example}
\newtheorem{problem}[theorem]{Problem}
\theoremstyle{remark}
\newtheorem{remark}[theorem]{Remark}

\newcommand{\Z}{\mathbb Z}
\newcommand{\Zthree}{\mathbb Z_3}
\newcommand{\Nzero}{\mathbb N_0}
\newcommand{\Rep}{R}
\newcommand{\vthree}{v_3}
\newcommand{\Cantor}{\Sigma_{3,\{0,1\}}}

\newcommand{\Adm}{\mathcal A}
\newcommand{\cell}{\mathcal B}
\newcommand{\tern}[1]{(#1)_3}

\title{\textbf{Repunit Coordinates for the Erdős Ternary-Digit Problem}}
\author{Michael A. Idowu}
\date{}

\begin{document}
\maketitle

\begin{abstract}
Paul Erd\H{o}s conjectured that the only powers of two whose ternary expansions contain no digit \(2\) are \(1\), \(4\), and \(256\). Despite extensive computational verification and substantial structural work, the conjecture remains open. This paper develops a self-contained reformulation based on generalised base-\(4\) repunits. Any admissible exponent is even, say \(n=2s\), and
\(
2^{2s}=4^s=3\Rep_s(4)+1,
\qquad
\Rep_s(4)=\frac{4^s-1}{3}.
\)
Thus \(2^{2s}\) omits the ternary digit \(2\) if and only if \(\Rep_s(4)\) does.
For every depth \(m\), we prove that
\(
\Phi_m:s\pmod{3^m}\longmapsto \Rep_s(4)\pmod{3^m}
\)
is a permutation of the residue classes modulo \(3^m\). Consequently, every prescribed \(m\)-digit ternary suffix has a unique canonical exponent seed modulo \(3^m\), and all exponents producing that suffix form an arithmetic progression. The central lifting identity
\(
\Rep_{r+q3^m}(4)
\equiv
\Rep_r(4)+q3^m
\pmod{3^{m+1}}
\)
shows that every residue class has exactly three lifts at the next depth and that exactly two extend an admissible suffix using only the digits \(0\) and \(1\). This places the trailing-digit recursion in an explicit repunit-coordinate framework and leads naturally to canonical seed fibres and an inverse-limit description.
Passing to the \(3\)-adic integers, the repunit map extends to an isometric homeomorphism that conjugates translation \(x\mapsto x+1\) with the affine recurrence \(y\mapsto4y+1\). The Erd\H{o}s problem is thereby separated into two distinct components: an exact and abundant finite-depth \(3\)-adic structure, and the unresolved global problem of determining which ordinary non-negative integer exponents satisfy the digit restriction at every ternary position. Worked examples, canonical seed tables, reproducible finite-depth datasets, and an explicit search protocol are included. No proof of the Erd\H{o}s conjecture is claimed.
\end{abstract}

\noindent\textbf{Keywords:} Erd\H{o}s conjecture; ternary digits; powers of two; generalised repunits; $3$-adic integers; residue classes; lifting; polynomial screening; digit omission.

\newpage

\section{Introduction}

In a short 1979 article, Erd\H{o}s conjectured that for $k>8$ the number $2^k$ cannot be represented as a sum of distinct powers of $3$, while recording the exceptional identity
\[
2^8=256=3^5+3^2+3+1.
\]
Equivalently, the base-$3$ expansion of $2^k$ contains no digit $2$ only for the known exponents $k=0,2,8$ \cite{Erdos1979}. The weaker formulation asks merely whether there are finitely many such exponents; this is the form emphasised in later dynamical and $3$-adic work \cite{Lagarias2009}. The problem is still open in 2026 \cite{Bloom406}.

The problem sits at an unusual interface. It is elementary to state, and any fixed exponent is easy to test, yet an argument excluding all future exponents has resisted both computation and known distribution theory. Gupta verified the conjecture through $n=4373$ \cite{Gupta1978}; later work substantially enlarged the computational range. In particular, Saye developed a recursive trailing-digit construction and verified the conjecture for all
\(
n\leq 2\cdot3^{45}\approx5.9\times10^{21}
\)
\cite{Saye2022}. \\
Dimitrov and Howe proved that among powers of two expressible as sums of at most twenty-five distinct powers of three, only the three known examples occur \cite{DimitrovHowe2025}. Lagarias developed real and $3$-adic dynamical formulations and connected the problem to Cantor-set intersections \cite{Lagarias2009}; Abram and Lagarias subsequently studied finite intersections of multiplicative translates of the $3$-adic Cantor set through finite automata and Hausdorff dimension \cite{AbramLagarias2014}. Dupuy and Weirich obtained average digit-equidistribution results whose non-averaged analogue would imply the weak Erd\H{o}s conjecture \cite{DupuyWeirich2016}. Aliyev has also analysed visible regularities in ternary digit blocks of powers of two \cite{Aliyev2023}.

The purpose of the present paper is different from another large computation and different from a claim of resolution. It develops a coordinate system in which the local arithmetic of the problem becomes exact and recursive. The presentation is deliberately self-contained: the repunit identities, congruence expansions, residue permutations, canonical seeds and arithmetic fibres are established from first principles before they are used. The resulting framework can therefore be read, checked and reproduced from the present paper alone.

The key observation is that the Erd\H{o}s condition is naturally a \emph{base-$4$ repunit condition}. Once this bridge is made, the repunit block identity and its $3$-adic truncations become directly applicable. The resulting finite-depth maps are not merely convenient: they are permutations of the residue rings $\Z/3^m\Z$. This leads to canonical seeds, disjoint fibres and an exact lift law. At the inverse-limit level, the finite residue maps extend to a natural 3-adic coordinate on the exponent space. Taken together, these results completely characterise finite-depth suffix admissibility in the repunit coordinate: every prescribed suffix has a unique exponent class, the admissible classes form an exact binary lifting tree, and the compatible finite maps assemble into an isometric \(3\)-adic coordinate. They do not, however, resolve the ordinary-integer problem. The unresolved Erdős conjecture asks which non-negative integers satisfy the digit restriction throughout the complete ternary expansion.

\subsection{Scope and novelty boundary}

The contribution of this paper is the following synthesis:
\begin{itemize}
    \item an exact reduction of the Erd\H{o}s ternary-digit condition to digit omission for the generalised repunit $R_s(4)$;
    \item a canonical seed-fibre stratification of exponent classes through the permutation
    \[
    \Phi_m(s)=R_s(4)\pmod{3^m};
    \]
    \item a repunit-coordinate formulation of the one-digit lifting law, in which each depth-$m$ class has three lifts and exactly two preserve admissibility;
    \item an inverse-limit formulation in which the compatible maps $\Phi_m$ assemble into an isometric homeomorphism of $\mathbb{Z}_3$ conjugating $x\mapsto x+1$ with $y\mapsto 4y+1$.
\end{itemize}

Taken together, these results provide a unified repunit-coordinate description of finite-depth admissibility, canonical inverse fibres, and the associated $3$-adic dynamics.

Several ingredients used in establishing this framework are classical or already present in the literature. Geometric-series identities, lifting-the-exponent arguments, multiplicative orders, and the $3$-adic Cantor set are not new. Most importantly, the recursive fact that there are three exponent lifts at each new ternary digit and that exactly two avoid the digit $2$ is already embodied in Saye's construction~\cite{Saye2022}. The contribution here is therefore not the discovery of that recursion itself, but its integration into a repunit-coordinate framework that also exposes the finite-level permutation structure, canonical seed fibres, and the compatible inverse-limit conjugacy on $\mathbb{Z}_3$.

A finite-depth residue tree, no matter how deep, cannot by itself prove finiteness of ordinary integer solutions, because the inverse-limit tree contains infinitely many $3$-adic paths. Any complete proof must therefore introduce a genuinely global mechanism beyond finite-depth suffix admissibility.

\begin{example}[Local admissibility versus global failure]
A small example already displays the distinction between the exact local
theory and the unresolved global problem. Take
\[
s=12.
\]
Then
\[
\Rep_{12}(4)
=
\frac{4^{12}-1}{3}
=
5\,592\,405
=
(101112010100010)_3.
\]
Modulo \(3^3=27\),
\[
\Rep_{12}(4)\equiv3=(010)_3.
\]
Thus the three-digit ternary suffix is \(010\), so all three visible
digits are admissible. Nevertheless, the complete ternary
expansion contains the digit \(2\).

Hence \(s=12\) survives the depth-\(3\) local test but is not a solution
of the Erd\H{o}s problem. This simple example captures the novelty
boundary of the present framework: the residue theory classifies local
admissibility exactly, but a separate global argument is needed to
exclude ordinary integers whose forbidden digit occurs further to the
left.
\end{example}

\section{Definitions and the Erd\H{o}s challenge}

\begin{definition}[Ternary digit-omission set]
Let
\[
\Cantor^{\mathrm{fin}}
:=\left\{\sum_{j=0}^{L-1}\varepsilon_j3^j:
L\geq0,\ \varepsilon_j\in\{0,1\}\right\}\subset\Nzero.
\]
Thus $a\in\Cantor^{\mathrm{fin}}$ exactly when the finite ternary expansion of $a$ omits the digit $2$.
\end{definition}

\begin{problem}[Sharp Erd\H{o}s ternary-digit problem]
Determine all $n\in\Nzero$ for which
\[
2^n\in\Cantor^{\mathrm{fin}}.
\]
The conjectured answer is $n\in\{0,2,8\}$ \cite{Erdos1979,Saye2022}.
\end{problem}

\begin{problem}[Weak finiteness form]
Prove that the set
\[
\{n\in\Nzero:2^n\in\Cantor^{\mathrm{fin}}\}
\]
is finite \cite{Lagarias2009,DupuyWeirich2016}.
\end{problem}

The distinction is useful. A proof that sufficiently large powers always contain a ternary digit $2$ would establish the weak form without identifying the exact exceptional set; a complete proof of the original sharp form must additionally eliminate every exponent above $8$.

\subsection{Generalised repunits and the arithmetic toolkit}

For an integer base $b>1$ and $s\geq0$, define
\[
\Rep_0(b)=0,
\qquad
\Rep_s(b)=\frac{b^s-1}{b-1}=1+b+\cdots+b^{s-1}\quad(s\geq1).
\]
Everything about repunits used later follows from this finite geometric sum. We record the required facts explicitly so that the later development is self-contained.

\begin{proposition}[Recurrence, addition, composition and blocks]\label{prop:repunit-toolkit}
For integers $b>1$ and non-negative integers $u,v,n$, one has
\begin{align}
\Rep_{n+1}(b)&=b\Rep_n(b)+1, \label{eq:recurrence}\\
\Rep_{u+v}(b)&=\Rep_u(b)+b^u\Rep_v(b), \label{eq:addition}
\end{align}
and, for $u\ge1$ and $v\ge0$,
\begin{equation}
\Rep_{uv}(b)=\Rep_u(b)\Rep_v(b^u). \label{eq:composition}
\end{equation}
Consequently, whenever $a\ge1$ and $n=aq+r$ with $q\ge0$ and $0\le r<a$,
\begin{equation}
\Rep_{aq+r}(b)=\Rep_a(b)\Rep_q(b^a)+b^{aq}\Rep_r(b).
\label{eq:block}
\end{equation}
\end{proposition}

\begin{proof}
The recurrence follows by multiplying the geometric sum for $\Rep_n(b)$ by $b$ and adjoining the missing constant term. Splitting the sum for $\Rep_{u+v}(b)$ after its first $u$ terms gives
\[
\sum_{j=0}^{u+v-1}b^j
=\sum_{j=0}^{u-1}b^j+b^u\sum_{j=0}^{v-1}b^j,
\]
which proves \eqref{eq:addition}. For composition,
\[
\Rep_u(b)\Rep_v(b^u)
=\left(\sum_{i=0}^{u-1}b^i\right)
 \left(\sum_{j=0}^{v-1}b^{uj}\right)
=\sum_{j=0}^{v-1}\sum_{i=0}^{u-1}b^{uj+i}.
\]
Every exponent from $0$ to $uv-1$ occurs exactly once, proving \eqref{eq:composition}. Finally, apply \eqref{eq:addition} to $aq+r$ and then \eqref{eq:composition} to $aq$.
\end{proof}

\begin{theorem}[Expansion about $1$ and finite $p$-adic truncation]\label{thm:padic-toolkit}
For every integer $s\ge1$,
\begin{equation}
\Rep_s(1+x)=\sum_{j=0}^{s-1}\binom{s}{j+1}x^j.
\label{eq:poly}
\end{equation}
More generally, if $p$ is prime, $e\ge1$ and $b=1+cp^e$ with $c\in\Z$, then for every $m\ge1$,
\begin{equation}
\Rep_s(b)\equiv
\sum_{j=0}^{m-1}\binom{s}{j+1}(cp^e)^j
\pmod{p^{em}},
\label{eq:generaltrunc}
\end{equation}
where $\binom{s}{j+1}=0$ when $j+1>s$.
\end{theorem}

\begin{proof}
The binomial theorem gives
\[
(1+x)^s-1=\sum_{k=1}^{s}\binom{s}{k}x^k.
\]
Dividing by $x$ and writing $j=k-1$ proves \eqref{eq:poly}. Substituting $x=cp^e$ gives an exact finite sum. Every omitted term with $j\ge m$ contains the factor $p^{ej}$ and is therefore divisible by $p^{em}$, proving \eqref{eq:generaltrunc}.
\end{proof}

For $b=4=1+3$, \cref{thm:padic-toolkit} becomes
\begin{equation}
\Rep_s(4)=\sum_{j=0}^{s-1}\binom{s}{j+1}3^j,
\label{eq:base4poly}
\end{equation}
and hence, for every $m\ge1$,
\begin{equation}
\Rep_s(4)\equiv
\sum_{j=0}^{m-1}\binom{s}{j+1}3^j
\pmod{3^m}.
\label{eq:base4trunc}
\end{equation}
This is an exact finite-depth formula. The binomial coefficients are not themselves ternary digits because carries remain present, but the identity gives a useful algebraic coordinate for the first $m$ digits.

\begin{example}[A numerical block decomposition and finite \(3\)-adic expansion]
Consider \(\Rep_7(4)\) and split the index as
\[
7=3\cdot2+1.
\]
The block identity gives
\[
\Rep_7(4)
=
\Rep_3(4)\Rep_2(4^3)
+
4^6\Rep_1(4).
\]
Now
\[
\Rep_3(4)=1+4+16=21,
\]
\[
\Rep_2(64)=1+64=65,
\]
and
\[
\Rep_1(4)=1.
\]
Therefore
\[
\Rep_7(4)
=
21\cdot65+4096
=
5461,
\]
which agrees with
\[
\Rep_7(4)=\frac{4^7-1}{3}=5461.
\]

The finite \(3\)-adic expansion can be seen just as concretely. For
\(s=5\), modulo \(27\),
\[
\Rep_5(4)
\equiv
5
+
3\binom52
+
3^2\binom53
\pmod{27}.
\]
Hence
\[
\Rep_5(4)
\equiv
5+30+90
=
125
\equiv17
\pmod{27}.
\]
Since
\[
17=(122)_3,
\]
the calculation recovers exactly the last three ternary digits of
\[
\Rep_5(4)=341=(110122)_3.
\]

This also illustrates why the binomial coefficients should not
themselves be interpreted as ternary digits: carries are essential.
\end{example}

\section{The repunit reduction}

The first theorem turns the ternary-digit problem into a repunit problem in a lossless way.

\begin{theorem}[Parity and repunit equivalence]\label{thm:repunit-equivalence}
Suppose the ternary expansion of $2^n$ contains only $0$ and $1$. Then $n$ is even. Writing $n=2s$,
\begin{equation}
2^{2s}=4^s=1+3\Rep_s(4),
\label{eq:bridge}
\end{equation}
and
\begin{equation}
2^{2s}\in\Cantor^{\mathrm{fin}}
\quad\Longleftrightarrow\quad
\Rep_s(4)\in\Cantor^{\mathrm{fin}}.
\label{eq:digit-equivalence}
\end{equation}
\end{theorem}

\begin{proof}
If the ternary expansion of $2^n$ omits $2$, its least significant ternary digit cannot be $0$ because $3\nmid2^n$; it must therefore be $1$. Hence $2^n\equiv1\pmod3$. Since $2\equiv-1\pmod3$, this forces $n$ even, say $n=2s$.

The identity \eqref{eq:bridge} is the definition of $\Rep_s(4)$. Since $4^s\equiv1\pmod3$, its least significant ternary digit is $1$. Subtracting $1$ from $4^s$ therefore causes no borrowing, and division by $3$ simply shifts the remaining ternary digits one place to the right. Thus the digits of $\Rep_s(4)$ are exactly the ternary digits of $4^s$ with the terminal $1$ removed. This proves \eqref{eq:digit-equivalence}.
\end{proof}

\begin{corollary}[Equivalent exponent formulation]\label{cor:sharp-s}
The sharp Erd\H{o}s conjecture is equivalent to
\[
\{s\in\Nzero:\Rep_s(4)\in\Cantor^{\mathrm{fin}}\}=\{0,1,4\}.
\]
\end{corollary}

\begin{example}[The three known solutions]
The repunit coordinate makes the known solutions transparent:
\[
\begin{array}{c|c|c|c}
 s & n=2s & \Rep_s(4) & 2^{2s}=1+3\Rep_s(4)\\
\hline
0&0&0&1=(1)_3\\
1&2&1=(1)_3&4=(11)_3\\
4&8&85=(10011)_3&256=(100111)_3.
\end{array}
\]
The intermediate value $s=2$ already fails because $\Rep_2(4)=5=(12)_3$.
\end{example}

\subsection{A valuation invariant}

The required valuation relation can be proved directly, without invoking a separate lifting-the-exponent source.

\begin{lemma}[Principal $3$-adic growth]\label{lem:principal-growth}
For every $t\ge0$,
\begin{equation}
4^{3^t}=1+3^{t+1}u_t,
\qquad u_t\equiv1\pmod3.
\label{eq:principal-growth}
\end{equation}
In particular, $\vthree(4^{3^t}-1)=t+1$.
\end{lemma}

\begin{proof}
For $t=0$, $4=1+3$, so the statement holds with $u_0=1$. Suppose
$4^{3^t}=1+3^{t+1}u_t$ with $u_t\equiv1\pmod3$. Cubing gives
\[
4^{3^{t+1}}
=(1+3^{t+1}u_t)^3
=1+3^{t+2}u_t+3^{2t+3}u_t^2+3^{3t+3}u_t^3.
\]
The last two terms are divisible by $3^{t+3}$, so the quotient of
$4^{3^{t+1}}-1$ by $3^{t+2}$ is congruent to $u_t\equiv1\pmod3$. This proves the induction.
\end{proof}

\begin{proposition}[Trailing zero valuation]\label{prop:valuation}
For every $s\ge1$,
\begin{equation}
\vthree\!\left(\Rep_s(4)\right)=\vthree(s).
\label{eq:valuation}
\end{equation}
\end{proposition}

\begin{proof}
Write $s=3^t u$ with $3\nmid u$. Put $X=4^{3^t}$. Then
\[
4^s-1=X^u-1=(X-1)(1+X+\cdots+X^{u-1}).
\]
By lemma \ref{lem:principal-growth}, $\vthree(X-1)=t+1$ and $X\equiv1\pmod3$. The second factor is therefore congruent to $u\not\equiv0\pmod3$, so it contributes no further factor of $3$. Thus
$\vthree(4^s-1)=t+1$. Since $\Rep_s(4)=(4^s-1)/3$, equation \eqref{eq:valuation} follows.
\end{proof}

Thus the number of trailing ternary zeros of $\Rep_s(4)$ is exactly the number of trailing base-$3$ zeros of $s$. This is the first indication that exponent digits and repunit digits are related by a triangular, valuation-compatible transformation rather than randomly.

\begin{example}[Trailing zeros are controlled exactly by the exponent]
Take
\[
s=18=2\cdot3^2.
\]
Thus
\[
\vthree(s)=2.
\]
The valuation theorem predicts
\[
\vthree\!\left(\Rep_{18}(4)\right)=2.
\]

Indeed, modulo \(27\), the truncated expansion gives
\[
\Rep_{18}(4)
\equiv
18
+
3\binom{18}{2}
+
9\binom{18}{3}
\pmod{27}.
\]
The second and third terms are divisible by \(27\), so
\[
\Rep_{18}(4)\equiv18\pmod{27}.
\]
Since
\[
18=2\cdot3^2,
\]
the residue is divisible by \(9\) but not by \(27\).

Therefore \(\Rep_{18}(4)\) has exactly two trailing ternary zeros.
Equivalently, its ternary expansion ends in a block of the form
\[
\cdots 1200_3.
\]

Thus multiplication of the exponent by a power of \(3\) has an exact
visible consequence at the right-hand end of the repunit expansion.
\end{example}

\section{Finite-depth canonical residue blocks}

We now organise the finite residue problem by unique congruence seeds and their complete arithmetic fibres. The construction is defined and proved directly in the present setting.

\begin{definition}[Finite repunit coordinate]
For $m\ge1$, define
\[
\Phi_m:\Z/3^m\Z\longrightarrow\Z/3^m\Z,
\qquad
\Phi_m([s])=[\Rep_s(4)].
\]
\end{definition}

The definition is well posed because exponent shifts by $3^m$ leave the repunit unchanged modulo $3^m$; a stronger statement will follow from the lift theorem in \cref{sec:lifting}. We first prove bijectivity directly.

\begin{lemma}[Order of $4$ modulo powers of $3$]\label{lem:order}
For $m\ge0$,
\begin{equation}
\operatorname{ord}_{3^{m+1}}(4)=3^m.
\label{eq:order4}
\end{equation}
\end{lemma}

\begin{proof}
By lemma \ref{lem:principal-growth}, $4^{3^m}\equiv1\pmod{3^{m+1}}$, while for $m\ge1$ the same lemma gives
$\vthree(4^{3^{m-1}}-1)=m$, so $4^{3^{m-1}}\not\equiv1\pmod{3^{m+1}}$.
Because $4\equiv1\pmod3$, its order modulo $3^{m+1}$ divides the $3$-power part of $\varphi(3^{m+1})=2\cdot3^m$. The only possibility is therefore $3^m$.
\end{proof}

\begin{theorem}[Permutation theorem]\label{thm:permutation}
For every $m\ge1$, $\Phi_m$ is a permutation of $\Z/3^m\Z$.
\end{theorem}

\begin{proof}
If $\Phi_m([r])=\Phi_m([s])$, then
\[
\frac{4^r-1}{3}\equiv\frac{4^s-1}{3}\pmod{3^m},
\]
so $4^r\equiv4^s\pmod{3^{m+1}}$. Therefore $4^{r-s}\equiv1\pmod{3^{m+1}}$. By lemma \ref{lem:order}, $3^m\mid r-s$, hence $[r]=[s]$ in $\Z/3^m\Z$. Thus $\Phi_m$ is injective and, on a finite set, bijective.
\end{proof}

\begin{definition}[Canonical seed and fibre]\label{def:seed-fibre}
For $c\in\{0,1,\dots,3^m-1\}$, let $\sigma_m(c)$ be the unique integer with
\[
0\le\sigma_m(c)<3^m,
\qquad
\Rep_{\sigma_m(c)}(4)\equiv c\pmod{3^m}.
\]
The depth-$m$ fibre over $c$ is
\[
\cell_m(c):=
\{s\in\Nzero:\Rep_s(4)\equiv c\pmod{3^m}\}.
\]
\end{definition}

\begin{corollary}[Exhaustive arithmetic fibres]\label{cor:fibres}
For every $m\ge1$ and residue $c$,
\begin{equation}
\cell_m(c)=\{\sigma_m(c)+3^mq:q\in\Nzero\}.
\label{eq:fibre}
\end{equation}
The fibres are disjoint and exhaust $\Nzero$.
\end{corollary}

\begin{proof}
The residue class of $s$ modulo $3^m$ is uniquely determined by $\Phi_m(s)=c$. The least non-negative representative is $\sigma_m(c)$.
\end{proof}

This is a canonical seed-plus-fibre decomposition specialised to the ternary-digit problem. There is one index modulus at each depth, and no global Diophantine conclusion is attached to the decomposition without an additional argument.

\subsection{Admissible suffixes}

For $m\ge1$, let
\begin{equation}
C_m:=\left\{\sum_{j=0}^{m-1}\varepsilon_j3^j:\varepsilon_j\in\{0,1\}\right\}
\subset\{0,\ldots,3^m-1\}.
\label{eq:Cm}
\end{equation}
Then $|C_m|=2^m$.

\begin{definition}[Admissible exponent classes]
Define
\begin{equation}
\Adm_m:=\Phi_m^{-1}(C_m)
=\{s\pmod{3^m}:\Rep_s(4)\bmod3^m\in C_m\}.
\label{eq:Am}
\end{equation}
\end{definition}

\begin{corollary}[Exact finite-depth count]\label{cor:count}
For every $m\ge1$,
\begin{equation}
|\Adm_m|=2^m,
\qquad
\frac{|\Adm_m|}{3^m}=\left(\frac23\right)^m.
\label{eq:density}
\end{equation}
\end{corollary}

\begin{proof}
By \cref{thm:permutation}, $\Phi_m$ is a bijection, so the inverse image of the $2^m$-element set $C_m$ has exactly $2^m$ elements.
\end{proof}

Equation \eqref{eq:density} is exact, not probabilistic. At depth $m$, precisely a proportion $(2/3)^m$ of exponent residue classes have their first $m$ repunit ternary digits in $\{0,1\}$. What it does \emph{not} say is how many finite ordinary exponents avoid $2$ in \emph{all} digits of the full repunit.

\begin{example}[A complete canonical fibre]
At depth \(m=3\), consider the admissible target residue
\[
c=(010)_3=3.
\]
The canonical seed is
\[
\sigma_3(3)=12.
\]
Indeed,
\[
\Rep_{12}(4)\equiv3\pmod{27}.
\]

The complete depth-\(3\) fibre is therefore
\[
\cell_3(3)
=
\{12+27q:q\in\Nzero\}.
\]
For example,
\[
s=12,\qquad s=39,\qquad s=66
\]
all produce repunits whose final three ternary digits are \(010\).
In particular,
\[
\Rep_{39}(4)\equiv3=(010)_3\pmod{27}.
\]

This example makes the word \emph{fibre} literal: a prescribed
three-digit output suffix determines not merely one exponent, but one
complete arithmetic progression of exponents.
\end{example}

\section{The one-digit lift law}\label{sec:lifting}

The internal repunit toolkit is especially effective at the transition from depth $m$ to $m+1$.

\begin{lemma}[A principal block]\label{lem:principalblock}
For $m\ge0$,
\begin{equation}
\Rep_{3^m}(4)\equiv3^m\pmod{3^{m+1}}.
\label{eq:principalblock}
\end{equation}
\end{lemma}

\begin{proof}
By lemma \ref{lem:principal-growth},
\[
4^{3^m}-1=3^{m+1}u_m,
\qquad u_m\equiv1\pmod3.
\]
Dividing by $3$ gives
\[
\Rep_{3^m}(4)=3^m u_m\equiv3^m\pmod{3^{m+1}},
\]
as required.
\end{proof}

\begin{theorem}[Repunit block lift]\label{thm:lift}
Let $m\ge0$, $0\le r<3^m$ and $q\ge0$. Then
\begin{equation}
\Rep_{r+q3^m}(4)
\equiv
\Rep_r(4)+q3^m
\pmod{3^{m+1}}.
\label{eq:lift}
\end{equation}
\end{theorem}

\begin{proof}
By the additive identity \eqref{eq:addition},
\[
\Rep_{r+q3^m}(4)
=\Rep_r(4)+4^r\Rep_{q3^m}(4).
\]
By the composition identity \eqref{eq:composition},
\[
\Rep_{q3^m}(4)=\Rep_{3^m}(4)\Rep_q(4^{3^m}).
\]
Now lemma \ref{lem:principalblock} gives
$\Rep_{3^m}(4)\equiv3^m\pmod{3^{m+1}}$. Also $4^{3^m}\equiv1\pmod3$, so
\[
\Rep_q(4^{3^m})\equiv q\pmod3.
\]
Finally $4^r\equiv1\pmod3$. Multiplying these congruences yields
\[
4^r\Rep_{q3^m}(4)\equiv q3^m\pmod{3^{m+1}},
\]
which proves the result.
\end{proof}

\begin{corollary}[Three lifts, one new digit]\label{cor:three-lifts}
For $t\in\{0,1,2\}$,
\begin{equation}
\Rep_{r+t3^m}(4)
\equiv\Rep_r(4)+t3^m\pmod{3^{m+1}}.
\label{eq:three-lifts}
\end{equation}
Hence the lower $m$ ternary digits are unchanged, while the next digit runs once through $0,1,2$ as $t$ runs through $0,1,2$.
\end{corollary}

\begin{corollary}[Perfect binary admissibility tree]\label{cor:binary-tree}
Every class in $\Adm_m$ has exactly two children in $\Adm_{m+1}$ among its three lifts
\[
r,
\quad r+3^m,
\quad r+2\cdot3^m
\pmod{3^{m+1}}.
\]
Consequently the admissible residue tree is a full rooted binary tree embedded in the ternary exponent tree.
\end{corollary}

\begin{proof}
The existing lower $m$ output digits lie in $\{0,1\}$. By corollary \ref{cor:three-lifts}, the new digit assumes $0$, $1$ and $2$ exactly once. The two choices yielding $0$ or $1$ are admissible and the third is not.
\end{proof}

\subsection{Relationship with Saye's recursion}

Saye proves an equivalent next-digit mechanism directly for powers of two. In his notation, the smallest positive exponent period preserving $k$ trailing ternary digits is $u_k=2\cdot3^{k-1}$, and the three lifts of an exponent through this period cycle the next ternary digit \cite{Saye2022}. Under the present reduction $n=2s$, the exponent increment
\[
2\cdot3^m
\]
in $n$ is exactly the increment $3^m$ in $s$. Thus \cref{thm:lift} and corollary \ref{cor:three-lifts} are a repunit-coordinate formulation of the same finite-depth arithmetic phenomenon.

This equivalence is useful rather than redundant. Saye's formulation is optimised for computation on $2^n$; the repunit formulation exposes an additional permutation structure, canonical inverse fibres and a natural inverse-limit map on $\Zthree$. In the present treatment the recursion follows directly from the block identities proved in proposition \ref{prop:repunit-toolkit}.

\begin{example}[The three lifts explicitly]
Start with
\[
r=4,
\qquad
m=2.
\]
Since
\[
\Rep_4(4)=85=(10011)_3,
\]
its final two ternary digits are \(11\).

The three exponent lifts modulo \(3^3=27\) are
\[
4,\qquad
4+9=13,\qquad
4+18=22.
\]
The lift theorem gives
\[
\Rep_4(4)\equiv4=(011)_3\pmod{27},
\]
\[
\Rep_{13}(4)\equiv13=(111)_3\pmod{27},
\]
and
\[
\Rep_{22}(4)\equiv22=(211)_3\pmod{27}.
\]

The old suffix \(11\) is unchanged, while the new ternary digit runs
through
\[
0,\qquad1,\qquad2.
\]
Exactly the first two lifts remain admissible.

Returning to the original power-of-two exponent \(n=2s\), the
corresponding exponents are
\[
8,\qquad26,\qquad44,
\]
separated by
\[
18=2\cdot3^2.
\]
This gives a concrete numerical bridge between the repunit lift and the
corresponding exponent-period formulation.
\end{example}

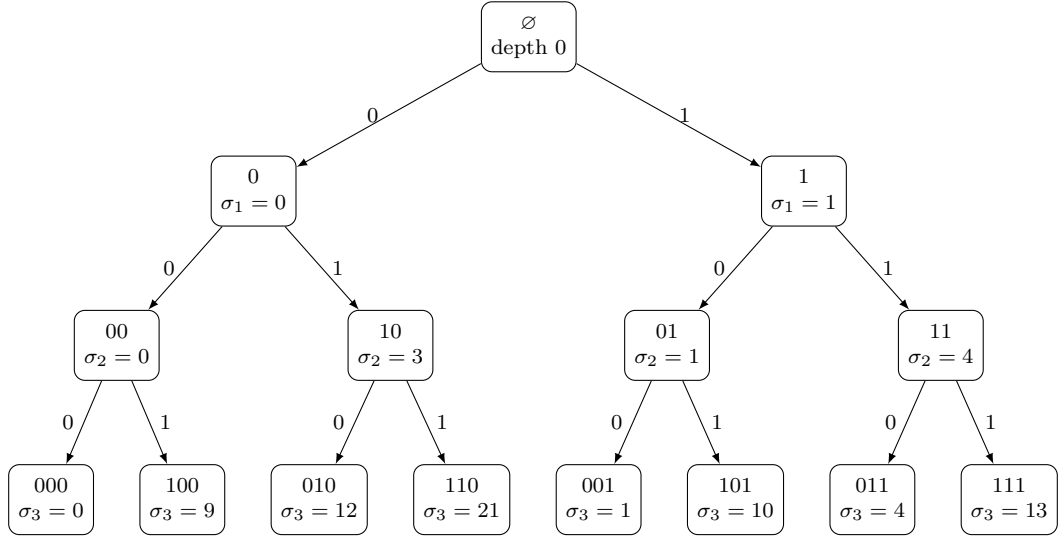
\begin{figure}[htbp]
\centering
\begin{forest}
for tree={
  draw,
  rounded corners,
  align=center,
  font=\scriptsize,
  l sep=11mm,
  s sep=6mm,
  edge={-latex},
}
[{$\varnothing$\\depth $0$}
  [{$0$\\$\sigma_1=0$}, edge label={node[midway,left,font=\scriptsize]{$0$}}
    [{$00$\\$\sigma_2=0$}, edge label={node[midway,left,font=\scriptsize]{$0$}}
      [{$000$\\$\sigma_3=0$}, edge label={node[midway,left,font=\scriptsize]{$0$}}]
      [{$100$\\$\sigma_3=9$}, edge label={node[midway,right,font=\scriptsize]{$1$}}]
    ]
    [{$10$\\$\sigma_2=3$}, edge label={node[midway,right,font=\scriptsize]{$1$}}
      [{$010$\\$\sigma_3=12$}, edge label={node[midway,left,font=\scriptsize]{$0$}}]
      [{$110$\\$\sigma_3=21$}, edge label={node[midway,right,font=\scriptsize]{$1$}}]
    ]
  ]
  [{$1$\\$\sigma_1=1$}, edge label={node[midway,right,font=\scriptsize]{$1$}}
    [{$01$\\$\sigma_2=1$}, edge label={node[midway,left,font=\scriptsize]{$0$}}
      [{$001$\\$\sigma_3=1$}, edge label={node[midway,left,font=\scriptsize]{$0$}}]
      [{$101$\\$\sigma_3=10$}, edge label={node[midway,right,font=\scriptsize]{$1$}}]
    ]
    [{$11$\\$\sigma_2=4$}, edge label={node[midway,right,font=\scriptsize]{$1$}}
      [{$011$\\$\sigma_3=4$}, edge label={node[midway,left,font=\scriptsize]{$0$}}]
      [{$111$\\$\sigma_3=13$}, edge label={node[midway,right,font=\scriptsize]{$1$}}]
    ]
  ]
]
\end{forest}
\caption{The first three levels of the admissible residue tree. A node label is the prescribed repunit suffix, written most-significant digit first, followed by its canonical seed $\sigma_m$. Each new edge prepends the next ternary digit $0$ or $1$. The omitted third lift would prepend $2$.}
\label{fig:tree}
\end{figure}

\section{The inverse-limit repunit coordinate}

The finite permutations $\Phi_m$ are compatible under reduction. This permits a global $3$-adic coordinate without initially invoking analytic $3$-adic exponentiation.

\begin{theorem}[$3$-adic repunit homeomorphism]\label{thm:homeomorphism}
There is a unique map
\[
\Phi:\Zthree\longrightarrow\Zthree
\]
whose reduction modulo $3^m$ is $\Phi_m$ for every $m\ge1$. The map $\Phi$ is an isometric homeomorphism: for $x,y\in\Zthree$,
\begin{equation}
\vthree\bigl(\Phi(x)-\Phi(y)\bigr)=\vthree(x-y),
\qquad x\ne y.
\label{eq:isometry}
\end{equation}
On non-negative integers,
\[
\Phi(s)=\Rep_s(4)=\frac{4^s-1}{3}.
\]
\end{theorem}

\begin{proof}
The congruence maps are compatible because $r\equiv s\pmod{3^m}$ implies, by \cref{thm:lift} or lemma \ref{lem:order},
\[
\Rep_r(4)\equiv\Rep_s(4)\pmod{3^m}.
\]
Thus the $\Phi_m$ define a map on the inverse limit
\[
\Zthree=\varprojlim_m \Z/3^m\Z.
\]
Each $\Phi_m$ is bijective, so the inverse maps are also compatible; hence $\Phi$ is a homeomorphism.

Moreover,
\[
x\equiv y\pmod{3^m}
\quad\Longleftrightarrow\quad
\Phi_m(x)=\Phi_m(y)
\quad\Longleftrightarrow\quad
\Phi(x)\equiv\Phi(y)\pmod{3^m},
\]
where the middle equivalence uses bijectivity of $\Phi_m$. The largest such $m$ is therefore the same on both sides, proving \eqref{eq:isometry}. Agreement with $\Rep_s(4)$ on ordinary integers follows from construction.
\end{proof}

In standard $3$-adic analytic notation, $\Phi(x)$ agrees with $(4^x-1)/3$, since $4=1+3$ admits $3$-adic exponentiation. The inverse-limit proof above is deliberately elementary and keeps the finite residue geometry visible.

\begin{proposition}[Affine conjugacy]\label{prop:conjugacy}
The homeomorphism $\Phi$ conjugates translation by $1$ to the affine map $y\mapsto4y+1$:
\begin{equation}
\Phi(x+1)=4\Phi(x)+1
\qquad(x\in\Zthree).
\label{eq:conjugacy}
\end{equation}
\end{proposition}

\begin{proof}
For $s\in\Nzero$, equation \eqref{eq:recurrence} gives
\[
\Rep_{s+1}(4)=4\Rep_s(4)+1.
\]
Ordinary integers are dense in $\Zthree$, and both sides of \eqref{eq:conjugacy} are continuous, so the identity extends to all $x\in\Zthree$.
\end{proof}

\begin{figure}[htbp]
\centering
\begin{tikzcd}[column sep=large,row sep=large]
\Zthree \arrow[r,"x\mapsto x+1"] \arrow[d,"\Phi"',"\cong"] &
\Zthree \arrow[d,"\Phi","\cong"'] \\
\Zthree \arrow[r,"y\mapsto4y+1"'] & \Zthree
\end{tikzcd}
\caption{Exact conjugacy between the $3$-adic adding machine in exponent coordinates and the base-$4$ repunit recurrence.}
\label{fig:conjugacy}
\end{figure}

The finite-level permutation structure therefore passes coherently to the inverse limit: the repunit coordinate defines an isometric homeomorphism of \(\mathbb Z_3\), and the recurrence \(R_{s+1}(4)=4R_s(4)+1\) becomes the affine conjugacy of proposition \ref{prop:conjugacy}. Having established the forward coordinate map and its finite-depth lifting structure, we now turn to the inverse problem: given a prescribed trailing ternary suffix, recover the unique canonical exponent seed that produces it.

\section{Canonical inverse construction}

The preceding finite and \(3\)-adic coordinate constructions naturally suggest the inverse problem: given a prescribed trailing ternary suffix of \(R_s(4)\), determine the unique exponent seed modulo the corresponding power of \(3\) that produces it.

\begin{proposition}[One-step inverse lift]\label{prop:inverse-lift}
Suppose $c_m\in\{0,\dots,3^m-1\}$ and $r=\sigma_m(c_m)$. Let $\delta_m\in\{0,1,2\}$ be the $(m+1)$st ternary digit of $\Rep_r(4)$, i.e.
\[
\delta_m\equiv \left\lfloor\frac{\Rep_r(4)}{3^m}\right\rfloor\pmod3.
\]
To prescribe a next digit $\varepsilon_m\in\{0,1,2\}$, choose the unique
\[
t\in\{0,1,2\}
\quad\text{with}\quad
\delta_m+t\equiv\varepsilon_m\pmod3.
\]
Then
\[
\sigma_{m+1}(c_m+\varepsilon_m3^m)=r+t3^m.
\]
\end{proposition}

\begin{proof}
This is corollary \ref{cor:three-lifts} read in reverse, together with uniqueness from \cref{thm:permutation}.
\end{proof}

For admissibility, one need not choose a fixed next digit: retain the two values of $t$ for which $\delta_m+t$ is $0$ or $1$ modulo $3$. This gives a breadth-first or depth-first search with exactly $2^m$ nodes at level $m$, matching the complexity of Saye's recursive construction \cite{Saye2022}.

\begin{remark}[Triangular digit transform]
Write
\[
s=\sum_{j\ge0}t_j3^j,
\qquad
\Phi(s)=\sum_{j\ge0}y_j3^j.
\]
The lift law implies that for each $m$ there is a function $\kappa_m$ of the lower exponent digits such that
\begin{equation}
y_m\equiv t_m+\kappa_m(t_0,\ldots,t_{m-1})\pmod3.
\label{eq:triangular}
\end{equation}
Thus $\Phi$ is a triangular digit permutation: the $m$th output digit depends on $t_m$ with coefficient $1$ and on lower digits through a carry term. This is another way of seeing why $\Phi_m$ is a permutation at every depth.
\end{remark}

\section{Illustrative datasets and reproducibility}

All tables in this section are reproducible by exact integer arithmetic. They are included to make the finite structure concrete rather than as evidence of a proof.

\subsection{Small powers of two}

\begin{longtable}{@{}r r l r c@{}}
\caption{Ternary expansions of $2^n$ for $0\le n\le20$.}\label{tab:smallpowers}\\
\toprule
$n$ & $2^n$ & $\tern{2^n}$ & number of digit $2$s & omits $2$?\\
\midrule
\endfirsthead
\toprule
$n$ & $2^n$ & $\tern{2^n}$ & number of digit $2$s & omits $2$?\\
\midrule
\endhead
0&1&\texttt{1}&0&Yes\\
1&2&\texttt{2}&1&No\\
2&4&\texttt{11}&0&Yes\\
3&8&\texttt{22}&2&No\\
4&16&\texttt{121}&1&No\\
5&32&\texttt{1012}&1&No\\
6&64&\texttt{2101}&1&No\\
7&128&\texttt{11202}&2&No\\
8&256&\texttt{100111}&0&Yes\\
9&512&\texttt{200222}&4&No\\
10&1024&\texttt{1101221}&2&No\\
11&2048&\texttt{2210212}&4&No\\
12&4096&\texttt{12121201}&3&No\\
13&8192&\texttt{102020102}&3&No\\
14&16384&\texttt{211110211}&2&No\\
15&32768&\texttt{1122221122}&6&No\\
16&65536&\texttt{10022220021}&5&No\\
17&131072&\texttt{20122210112}&5&No\\
18&262144&\texttt{111022121001}&3&No\\
19&524288&\texttt{222122012002}&7&No\\
20&1048576&\texttt{1222021101011}&4&No\\
\bottomrule
\end{longtable}

The table illustrates the sharp conjecture at small scale: only $n=0,2,8$ survive. It also shows why looking for a naive monotone digit statistic is unlikely to be sufficient: the number of $2$s fluctuates strongly.

\subsection{Finite-depth class counts}

\begin{table}[htbp]
\centering
\caption{Exact number of admissible exponent classes at depth $m$.}
\label{tab:counts}
\begin{tabular}{@{}rrrr@{}}
\toprule
$m$ & modulus $3^m$ & admissible classes $2^m$ & exact fraction $(2/3)^m$\\
\midrule
1&3&2&0.666667\\
2&9&4&0.444444\\
3&27&8&0.296296\\
4&81&16&0.197531\\
5&243&32&0.131687\\
6&729&64&0.087791\\
7&2187&128&0.058528\\
8&6561&256&0.039018\\
9&19683&512&0.026012\\
10&59049&1024&0.017342\\
\bottomrule
\end{tabular}
\end{table}

The exponentially decreasing fraction is compatible with a random-digit heuristic, but the exact count comes from bijectivity and is much stronger than a heuristic. Conversely, because the absolute count is $2^m$, the inverse-limit set remains rich.

\subsection{Canonical seeds for short admissible suffixes}

\begin{table}[htbp]
\centering
\caption{Canonical seeds $\sigma_m(c)$ for all admissible suffixes through depth $m=3$. The suffix is written in ordinary ternary order, most-significant digit first.}
\label{tab:seeds3}
\begin{tabular}{@{}ccrr@{}}
\toprule
$m$ & suffix $c$ & decimal $c$ & $\sigma_m(c)$\\
\midrule
1&\texttt{0}&0&0\\
1&\texttt{1}&1&1\\
\midrule
2&\texttt{00}&0&0\\
2&\texttt{01}&1&1\\
2&\texttt{10}&3&3\\
2&\texttt{11}&4&4\\
\midrule
3&\texttt{000}&0&0\\
3&\texttt{001}&1&1\\
3&\texttt{010}&3&12\\
3&\texttt{011}&4&4\\
3&\texttt{100}&9&9\\
3&\texttt{101}&10&10\\
3&\texttt{110}&12&21\\
3&\texttt{111}&13&13\\
\bottomrule
\end{tabular}
\end{table}

The seed need not look like the target suffix: for example, $\sigma_3((010)_3)=12=(110)_3$. The map is triangular but non-trivial because lower-digit carries modify higher output digits.

\begin{table}[htbp]
\centering
\caption{Depth-$4$ admissible target suffixes and their canonical exponent seeds.}
\label{tab:seeds4}
\begin{tabular}{@{}cr@{\qquad}cr@{}}
\toprule
suffix & seed & suffix & seed\\
\midrule
\texttt{0000}&0&\texttt{1000}&27\\
\texttt{0001}&1&\texttt{1001}&28\\
\texttt{0010}&12&\texttt{1010}&39\\
\texttt{0011}&4&\texttt{1011}&31\\
\texttt{0100}&36&\texttt{1100}&63\\
\texttt{0101}&10&\texttt{1101}&37\\
\texttt{0110}&48&\texttt{1110}&75\\
\texttt{0111}&13&\texttt{1111}&40\\
\bottomrule
\end{tabular}
\end{table}

\begin{example}[Output digits and exponent digits need not agree]
Consider the target suffix
\[
(010)_3=3.
\]
Its depth-\(3\) canonical seed is
\[
\sigma_3(3)=12=(110)_3.
\]
Hence
\[
\Rep_{12}(4)\equiv(010)_3\pmod{27}.
\]

By contrast,
\[
(011)_3=4
\]
has the simpler seed
\[
\sigma_3(4)=4=(011)_3.
\]

Thus the inverse residue map sometimes appears to preserve the ternary
word and sometimes rearranges it substantially. The difference is
caused by higher-place carries: the seed map is triangular, not
digitwise identical.
\end{example}

\subsection{Record minima for long admissible trailing blocks}

Following the finite-depth viewpoint used by Saye, define $\rho_2(m)$ as the smallest positive exponent $n$ for which $2^n$ has at least $m$ ternary digits and its last $m$ ternary digits omit $2$. Direct exact computation gives the initial values in \cref{tab:rho}; they agree with the sequence reported in Saye's computational study \cite{Saye2022}.

\begin{table}[htbp]
\centering
\caption{Initial record minima $\rho_2(m)$ and the corresponding admissible trailing block.}
\label{tab:rho}
\begin{tabular}{@{}r r l@{}}
\toprule
$m$ & $\rho_2(m)$ & last $m$ ternary digits of $2^{\rho_2(m)}$\\
\midrule
1&2&\texttt{1}\\
2&2&\texttt{11}\\
3&6&\texttt{101}\\
4&8&\texttt{0111}\\
5&8&\texttt{00111}\\
6&8&\texttt{100111}\\
7&20&\texttt{1101011}\\
8&24&\texttt{01000101}\\
9&24&\texttt{101000101}\\
10&24&\texttt{0101000101}\\
11&72&\texttt{10011101001}\\
12&186&\texttt{011100100101}\\
13&186&\texttt{0011100100101}\\
14&332&\texttt{10000011000111}\\
15&332&\texttt{110000011000111}\\
16&1134&\texttt{0011000101100001}\\
\bottomrule
\end{tabular}
\end{table}

These minima grow irregularly. They concern only trailing blocks, not full digit omission. For example, a power may possess a very long admissible suffix while containing a $2$ further to the left.

\subsection{Reproducibility protocol}

The finite computations require only exact integer arithmetic:
\begin{enumerate}
  \item compute $\Rep_s(4)=(4^s-1)/3$ or, modulo $3^m$, use the recurrence $R_{s+1}\equiv4R_s+1$;
  \item reduce $\Rep_s(4)$ modulo $3^m$ and test whether every one of the $m$ ternary digits lies in $\{0,1\}$;
  \item for canonical inversion, start with the unique seed at depth $m$ and apply proposition \ref{prop:inverse-lift};
  \item independently verify the relation to powers of two using $2^{2s}=1+3\Rep_s(4)$;
  \item for large depths, traverse only the two admissible children of each node rather than all $3^m$ exponent classes.
\end{enumerate}

The last step reduces a depth-$m$ suffix enumeration from $\Theta(3^m)$ naive classes to $\Theta(2^m)$ nodes, consistent with Saye's recursive approach \cite{Saye2022}.

\begin{example}[Generating the first three admissible levels by hand]
Beginning with the depth-\(0\) seed \(0\), the admissible
representatives evolve as follows:
\[
\Adm_0=\{0\},
\]
\[
\Adm_1=\{0,1\},
\]
\[
\Adm_2=\{0,1,3,4\},
\]
and
\[
\Adm_3
=
\{0,1,4,9,10,12,13,21\}.
\]
Their cardinalities are
\[
1,\qquad2,\qquad4,\qquad8,
\]
as predicted.

For example, the depth-\(2\) representative \(r=4\) has the three lifts
\[
4,\qquad13,\qquad22.
\]
The representatives \(4\) and \(13\) survive, whereas \(22\) is
discarded because its newly exposed repunit digit is \(2\).

This provides a compact hand-check of the search algorithm before it is
implemented at large depth.
\end{example}

\section{Block decomposition beyond one digit}

The one-digit lift law is the smallest instance of a broader block calculation. Let $a=3^h$ and write
\[
s=aq+r,
\qquad0\le r<a.
\]
The block identity \eqref{eq:block} gives
\begin{equation}
\Rep_s(4)
=\Rep_a(4)\Rep_q(4^a)+4^{aq}\Rep_r(4).
\label{eq:block4}
\end{equation}
Because $a$ is a power of $3$, both $\Rep_a(4)$ and $4^a$ have controlled $3$-adic proximity to $0$ and $1$, respectively. In particular,
\[
\vthree(\Rep_{3^h}(4))=h,
\qquad
4^{3^h}\equiv1\pmod{3^{h+1}}.
\]
Thus a block of $h$ exponent digits naturally controls a corresponding block of output precision.

This offers a family of implementation strategies. One may lift one digit at a time, as in \cref{thm:lift}; or group several ternary exponent digits and precompute block transitions. The asymptotic number of surviving states remains $2^m$ at depth $m$, so blocking does not change the fundamental exponential search tree. Its value is architectural: vectorised transitions, cache locality, parallel search, and the possibility of attaching stronger global filters to whole residue blocks.

\subsection{Polynomial screening}

The truncated repunit expansion
\[
\Rep_s(4)\equiv
s+3\binom{s}{2}+3^2\binom{s}{3}+\cdots+3^{m-1}\binom{s}{m}
\pmod{3^m}
\]
provides a second exact way to compute the first $m$ ternary digits from $s$. For fixed $m$, this is a polynomial map on residue classes. The permutation theorem shows that despite its increasingly complicated carry structure, this polynomial induces a bijection modulo $3^m$.

A useful research question is whether the binomial form can support estimates unavailable from the digit tree alone: for example, bounds on carry propagation, correlations between distant output digits, or restrictions when $s$ is constrained to a short Archimedean interval. Such estimates would begin to couple the local $3$-adic coordinate with ordinary size.

\begin{example}[Computing a suffix without constructing \(4^{12}\)]
Suppose that only the last three ternary digits of \(\Rep_{12}(4)\) are
required. The truncated polynomial gives
\[
\Rep_{12}(4)
\equiv
12
+
3\binom{12}{2}
+
9\binom{12}{3}
\pmod{27}.
\]
Therefore
\[
\Rep_{12}(4)
\equiv
12+198+1980
=
2190
\equiv3
\pmod{27}.
\]
Since
\[
3=(010)_3,
\]
the final three ternary digits are \(010\).

The large integer \(4^{12}\) never needs to be constructed. This is the
practical meaning of polynomial screening: finite ternary precision can
be obtained from a polynomial in the exponent residue.
\end{example}

\section{Consequences of the framework}

The preceding theorems give several immediate structural conclusions.

\subsection{No finite suffix obstruction can resolve the conjecture}

For every \(m\geq 1\), there are exactly \(2^m\) exponent classes modulo \(3^m\) for which the least significant \(m\) ternary digits of \(R_s(4)\) all lie in \(\{0,1\}\). Moreover, by corollary \ref{cor:binary-tree}, every admissible class has exactly two admissible children at the next depth.\\
Consequently, the finite-depth admissibility tree never terminates. There is therefore no fixed depth \(m\) at which all residue classes are eliminated, and no argument based solely on checking a bounded number of trailing ternary digits can resolve the Erdős conjecture.

What remains unresolved is a genuinely global question: whether an ordinary non-negative integer \(s\) can satisfy the digit restriction throughout the complete ternary expansion of \(R_s(4)\), rather than merely through any prescribed finite suffix.

\subsection{The exact local density is already determined}

The finite-depth density is exactly $(2/3)^m$. Therefore a heuristic treating successive admissibility constraints as reducing the class count by a factor $2/3$ happens to predict the correct count, but not because the digits are independent. The exact reason is that $\Phi_m$ is a permutation and the target set has cardinality $2^m$.

\subsection{Sparse-support results impose a global threshold}

If the sharp conjecture were false, a counterexample $2^n$ would be a sum of distinct powers of $3$. Dimitrov and Howe prove that any counterexample beyond $1,4,256$ would require at least twenty-six non-zero ternary digits \cite{DimitrovHowe2025}. Since
\[
2^{2s}=3R_s(4) + 1,
\]
the terminal digit \(1\) of \(2^{2s}\) is contributed by the added constant term, while the remaining ternary digits are those of \(R_s(4)\), shifted one place to the left. Hence any new counterexample would force \(R_s(4)\) to contain at least twenty-five \(1\)-digits. This is a genuinely global restriction that is invisible from any bounded ternary suffix.

\begin{example}[From ternary support to an ordinary size restriction]
Suppose a new counterexample requires at least twenty-six non-zero
ternary digits in \(2^{2s}\). Since
\[
2^{2s}=3\Rep_s(4) + 1,
\]
this forces at least twenty-five \(1\)-digits in \(\Rep_s(4)\).

The smallest non-negative integer containing twenty-five \(1\)-digits
in base \(3\) is obtained by placing them in the lowest twenty-five
positions:
\[
1+3+\cdots+3^{24}
=
\frac{3^{25}-1}{2}.
\]
Consequently any such repunit must satisfy
\[
\Rep_s(4)
\ge
\frac{3^{25}-1}{2}.
\]
Using
\[
4^s=1+3\Rep_s(4),
\]
one obtains
\[
4^s
\ge
1+\frac{3(3^{25}-1)}{2}
=
\frac{3^{26}-1}{2}.
\]
Therefore
\[
s
\ge
\left\lceil
\log_4\!\left(\frac{3^{26}-1}{2}\right)
\right\rceil
=
21.
\]

This elementary bound is small compared with existing computational
verification, but it illustrates an important principle: a global
support theorem can be converted into an ordinary-height restriction.
\end{example}

\section{Further Research}

The framework is intended to sharpen the next mathematical question rather than obscure it. 

\subsection{Target A: an Archimedean escape theorem}

A direct route would be to prove that there exists $S_0$ such that every ordinary integer $s>S_0$ has
\[
\Phi(s)=\Rep_s(4)\notin\Cantor.
\]
Such a theorem must use more than residue cylinders because every cylinder contains infinitely many ordinary integers.

Possible forms of a missing global estimate include:
\begin{itemize}
  \item a lower bound on the number of ternary digit $2$s in $\Rep_s(4)$ as a function of $s$;
  \item a carry-propagation theorem for the binomial expansion \eqref{eq:base4poly};
  \item a Diophantine approximation inequality incompatible with long or complete membership in $\Cantor$;
  \item an entropy or discrepancy statement strong enough to hold for each sufficiently large exponent rather than only on average.
\end{itemize}

The last distinction is exactly the gap highlighted by Dupuy and Weirich: average equidistribution is accessible, whereas a sufficiently strong non-averaged statement would imply the Erd\H{o}s conjecture \cite{DupuyWeirich2016}.

\begin{example}[What an Archimedean escape theorem would need to prove]
Consider the depth-\(4\) fibre
\[
s=4+81q.
\]
Every member of this fibre has the same four trailing repunit digits as
the genuine solution \(s=4\).

For \(q=0\),
\[
s=4,
\]
and
\[
\Rep_4(4)=85=(10011)_3,
\]
which is globally admissible.

For \(q=1\),
\[
s=85.
\]
The repunit still possesses the same protected depth-\(4\) suffix, but
its complete ternary expansion contains forbidden digits \(2\).

An Archimedean escape theorem would establish uniformly that, once
\(q\) is sufficiently large, some ternary digit beyond the protected
four-digit suffix must equal \(2\). This is exactly the type of
information that purely local residue lifting cannot provide.
\end{example}

\subsection{Target B: block filters that couple both metrics}

A computationally oriented theorem target is to attach an Archimedean interval to each depth-$m$ fibre
\[
s=\sigma_m(c)+3^mq
\]
and prove that beyond a controlled range of $q$ some more significant ternary digit of $\Rep_s(4)$ must equal $2$. The block identity \eqref{eq:block4} is naturally suited to this because it separates the bounded seed $r$ from the large quotient $q$.

The challenge is that $4^{3^m}$ is extremely close to $1$ in the $3$-adic metric while enormous in the ordinary metric. This tension is precisely the kind of local-global interaction a successful proof must exploit.

\begin{example}[One quantity, two very different metric behaviours]
Take
\[
s=9q+4.
\]
For \(q=1\), one has \(s=13\), and the block identity gives
\[
\Rep_{13}(4)
=
\Rep_9(4)\Rep_1(4^9)
+
4^9\Rep_4(4).
\]
Now
\[
\Rep_9(4)=87381,
\]
\[
4^9=262144,
\]
and
\[
\Rep_4(4)=85.
\]
Therefore
\[
\Rep_{13}(4)
=
87381+262144\cdot85
=
22369621.
\]

On the other hand,
\[
4^9=262144\equiv1\pmod{27}.
\]
Thus \(4^9\) is extremely simple from the depth-\(3\) \(3\)-adic
viewpoint while already being very large in the ordinary metric.

This is precisely the local-global tension that an effective block
filter would have to exploit.
\end{example}

\subsection{Target C: finite automata plus height}

Abram and Lagarias show that intersections of multiplicative translates of the $3$-adic Cantor set can be described by finite automata \cite{AbramLagarias2014}. The repunit conjugacy suggests asking whether selected additional constraints on $s$ or $\Phi(s)$ can be encoded by an automaton whose accepted $3$-adic paths admit a useful Archimedean height estimate. Automata alone will retain infinite paths; the additional height mechanism is essential.

\begin{example}[Why an automaton needs a height coordinate]
Suppose an automaton follows the successive admissible target suffixes
\[
1,\qquad
11,\qquad
111,\qquad
1111.
\]
Every prefix is accepted, and the associated least exponent
representatives are
\[
1,\qquad
4,\qquad
13,\qquad
40.
\]

From a purely automata-theoretic viewpoint this is simply an indefinitely
accepted path. The automaton has no local reason to reject it.

A height annotation, however, records that the least ordinary
representative grows while the residue information becomes increasingly
precise. The intended research question is whether some additional
quantity - ordinary magnitude, ternary digit length, support count or
Diophantine height - can eventually become incompatible with indefinite
acceptance.

The automaton therefore supplies the local language but not, by itself,
the global obstruction.
\end{example}

\subsection{Target D: integrate sparse-support thresholds}

Dimitrov-Howe's result supplies a hard lower bound on the Hamming weight of any new ternary counterexample \cite{DimitrovHowe2025}. It is natural to enrich the residue tree with a running lower bound on the number of forced $1$s and combine it with size estimates for $4^s$. Even if this does not prove the full conjecture, it may produce rigorous exclusions for structured subfamilies of exponent fibres.

A further possible direction is to study the carry corrections $\kappa_m$ arising in the triangular digit representation (equation ~\ref{eq:triangular}). Although finite-depth admissibility exhibits no additional scarcity beyond the exact factor $(2/3)^m$, correlations among these correction terms across separated scales may encode information that is specific to ordinary integer exponents rather than arbitrary $3$-adic paths. Such correlations could therefore provide a mechanism for coupling the exact local residue structure to genuinely global restrictions. We do not pursue this direction further here.

\section{Discussion}

The repunit reduction changes the perspective in three ways.

First, it removes the superficial asymmetry between a power of two and a sum of powers of three. After the forced parity reduction, the problem concerns a single recurrence
\[
R_{s+1}=4R_s+1,
\qquad R_0=0,
\]
with the requirement that every ternary digit of $R_s$ lie in $\{0,1\}$. The target set is the standard ternary Cantor set; the orbit is an affine recurrence.

Secondly, the map from exponent residues to repunit residues is as regular as possible: it is a permutation at every $3$-power modulus. This fact explains exactly why two of three lifts survive at each digit. It also means there is no purely local congruence scarcity waiting to be discovered; thus any proof must incorporate information not contained in finite-depth \(3\)-adic suffix admissibility, naturally pointing towards constraints involving ordinary size or other global structure.

Thirdly, the construction places computation and proof in different roles. Deep recursion is excellent for verification and for discovering record admissible suffixes, as demonstrated at extraordinary depth by Saye \cite{Saye2022}. But the infinite binary tree proves that further depth alone cannot become a proof. The research task is to find a theorem that cuts across the tree using ordinary magnitude, sparse-support constraints, or a non-averaged distribution principle.

This separation is valuable for a foundational programme. It identifies which statements are exact, which are computational, and where the unsolved content begins.

\section{Conclusion}

The Erd\H{o}s ternary-digit challenge can be reformulated exactly through the base-$4$ repunit
\[
\Rep_s(4)=\frac{4^s-1}{3}.
\]
Every admissible power of two has even exponent $n=2s$, and digit omission for $2^{2s}$ is equivalent to digit omission for $\Rep_s(4)$. The general block identity and finite $p$-adic expansion required for this coordinate are proved explicitly in proposition \ref{prop:repunit-toolkit} and \cref{thm:padic-toolkit}.

At every depth $m$, the repunit residue map $\Phi_m(s)=\Rep_s(4)\bmod3^m$ is a permutation. Consequently, each target suffix has a unique canonical exponent seed modulo $3^m$, its full source set is an arithmetic fibre, and the admissible suffixes determine exactly $2^m$ classes. The block lift
\[
\Rep_{r+q3^m}(4)\equiv\Rep_r(4)+q3^m\pmod{3^{m+1}}
\]
explains the three-to-two branching law and recovers Saye's trailing-digit recursion in repunit coordinates.\\
The inverse-limit map \(\Phi:\mathbb Z_3\to\mathbb Z_3\) is an isometric homeomorphism and conjugates \(x\mapsto x+1\) with \(y\mapsto 4y+1\).\\
This is the central conceptual outcome: the local \(3\)-adic problem is completely populated and exactly recursive, whereas the unresolved content is global. A complete proof must explain why, among ordinary non-negative integers \(s\), only \(s=0,1,4\) produce repunits \(R_s(4)\) whose full ternary expansions contain no digit \(2\). The framework therefore supplies a rigorous finite-depth structure, a reproducible computational architecture, and a precise boundary for future theorem development, but it does not resolve the Erdős conjecture.

\appendix
\section{A compact computational algorithm}

A depth-first enumeration of admissible seeds may be stated as follows.

\begin{verbatim}
AdmissibleSeeds(M):
    states := {(0, 0)}   # (depth m, seed r modulo 3^m)
    leaves := {0}
    for m from 0 to M-1:
        next := empty set
        for r in leaves:
            d := floor(R_r(4) / 3^m) mod 3
            for t in {0,1,2}:
                if (d + t) mod 3 is in {0,1}:
                    next.add(r + t*3^m)
        leaves := next
    return leaves
\end{verbatim}

A modular implementation need not construct $R_r(4)$ in full. It is enough to maintain $R_r(4)\bmod3^{m+1}$, or to compute it by the binary power-and-sum recurrence given in \cref{app:binary}. The invariant after level $m$ is that \texttt{leaves} contains exactly the $2^m$ representatives of $\Adm_m$ in $[0,3^m)$.

\section{A depth-$2$ permutation example}

Modulo $9$, the full map $s\mapsto\Rep_s(4)$ is already non-trivial but bijective:
\[
\begin{array}{c|ccccccccc}
(s)_3&00&01&02&10&11&12&20&21&22\\
\hline
(\Rep_s(4)\bmod9)_3&00&01&12&10&11&22&20&21&02.
\end{array}
\]
The four admissible output residues $00,01,10,11$ correspond to the four exponent classes $00,01,10,11$ at this depth, but this superficial alignment disappears at the next level; for example the output suffix $010$ has seed $s=12=(110)_3$.

\section{Binary power-and-sum evaluation}\label{app:binary}

For completeness, the modular repunit values used by the search may be obtained in logarithmic recursion depth without division. For a fixed base $B$, write
\[
P_n:=B^n,
\qquad
S_n:=\Rep_n(B).
\]
Then
\begin{align}
P_{2k}&=P_k^2,
& S_{2k}&=S_k(1+P_k), \label{eq:binary-even}\\
P_{2k+1}&=B P_k^2,
& S_{2k+1}&=S_k(1+P_k)+P_k^2. \label{eq:binary-odd}
\end{align}
Indeed, \eqref{eq:binary-even} follows by splitting the geometric sum into two blocks of length $k$, while \eqref{eq:binary-odd} appends the term $B^{2k}=P_k^2$. Reducing every intermediate product modulo a chosen modulus $M$ gives $B^n\bmod M$ and $\Rep_n(B)\bmod M$ using $O(\log n)$ recursive levels. This is sufficient for all finite-depth computations in the paper.



\begin{thebibliography}{99}

\bibitem{Erdos1979}
P.~Erd\H{o}s,
``Some Unconventional Problems in Number Theory,''
\textit{Mathematics Magazine},
vol.~52, no.~2, pp.~67-70, 1979.
doi: 10.2307/2689842.

\bibitem{Gupta1978}
H.~Gupta,
``Powers of 2 and Sums of Distinct Powers of 3,''
\textit{Publikacije Elektrotehni\v{c}kog fakulteta. Serija Matematika i fizika},
no.~602-633, pp.~151-158, 1978.

\bibitem{Narkiewicz1980}
W.~Narkiewicz,
``A Note on a Paper of H. Gupta Concerning Powers of 2 and 3,''
\textit{Publikacije Elektrotehni\v{c}kog fakulteta. Serija Matematika i fizika},
no.~678-715, pp.~173-174, 1980.

\bibitem{Lagarias2009}
J.~C.~Lagarias,
``Ternary Expansions of Powers of 2,''
\textit{Journal of the London Mathematical Society},
vol.~79, no.~3, pp.~562-588, 2009.
doi: 10.1112/jlms/jdn080.

\bibitem{AbramLagarias2014}
W.~C.~Abram and J.~C.~Lagarias,
``Intersections of Multiplicative Translates of 3-adic Cantor Sets,''
\textit{Journal of Fractal Geometry},
vol.~1, no.~4, pp.~349-390, 2014.
doi: 10.4171/JFG/11.

\bibitem{DupuyWeirich2016}
T.~Dupuy and D.~E.~Weirich,
``Bits of $3^n$ in Binary, Wieferich Primes and a Conjecture of Erd\H{o}s,''
\textit{Journal of Number Theory},
vol.~158, pp.~268-280, 2016.
doi: 10.1016/j.jnt.2015.05.022.

\bibitem{Saye2022}
R.~I.~Saye,
``On Two Conjectures Concerning the Ternary Digits of Powers of Two,''
\textit{Journal of Integer Sequences},
vol.~25, no.~3, Article 22.3.4, 2022.
\url{https://cs.uwaterloo.ca/journals/JIS/VOL25/Saye/saye3.html}.

\bibitem{Aliyev2023}
Y.~N.~Aliyev,
``Digits of Powers of 2 in Ternary Numeral System,''
\textit{Notes on Number Theory and Discrete Mathematics},
vol.~29, no.~3, pp.~474-485, 2023.
doi: 10.7546/nntdm.2023.29.3.474-485.

\bibitem{DimitrovHowe2025}
V.~S.~Dimitrov and E.~W.~Howe,
``Powers of 3 with Few Nonzero Bits and a Conjecture of Erd\H{o}s,''
\textit{Rocky Mountain Journal of Mathematics},
vol.~55, no.~1, pp.~45-61, 2025.
doi: 10.1216/rmj.2025.55.45.

\bibitem{Bloom406}
T.~F.~Bloom,
``Erd\H{o}s Problem \#406,''
2026.
\url{https://www.erdosproblems.com/406}.
Accessed 21 September 2026.

\bibitem{HardyWright2008}
G.~H.~Hardy and E.~M.~Wright,
\textit{An Introduction to the Theory of Numbers},
6th ed.
Oxford: Oxford University Press, 2008.

\bibitem{IrelandRosen1990}
K.~Ireland and M.~Rosen,
\textit{A Classical Introduction to Modern Number Theory},
2nd ed.
New York: Springer, 1990.

\end{thebibliography}
\end{document}